\documentclass[conference]{IEEEtran}
\IEEEoverridecommandlockouts
\usepackage[USenglish,american]{babel}
\usepackage{cite}
\usepackage{amsmath,amssymb,amsfonts}
\usepackage{amsthm}
\usepackage{commath}
\usepackage{algorithmic}
\usepackage{graphicx}
\usepackage[inline]{enumitem}
\usepackage{url}
\usepackage{array}
\usepackage{textcomp}
\usepackage{xcolor}
\usepackage[final]{pdfpages}
\usepackage[bookmarks=false]{hyperref}
\PassOptionsToPackage{bookmarks=false}{hyperref}
\def\BibTeX{{\rm B\kern-.05em{\sc i\kern-.025em b}\kern-.08em
    T\kern-.1667em\lower.7ex\hbox{E}\kern-.125emX}}
    
\theoremstyle{definition}
\newtheorem{theorem}{Theorem}

\theoremstyle{definition}

\theoremstyle{definition}

\theoremstyle{definition}

\theoremstyle{definition}

\theoremstyle{definition}

\begin{document}

\title{Singular Perturbation Theory for a Finite-Dimensional, Discrete-Time Chi Nonlinear System}

\author{\IEEEauthorblockN{Xiaofan Cui and Al-Thaddeus Avestruz}
\thanks{*Xiaofan Cui and Al-Thaddeus Avestruz are with the Department
of Electrical Engineering and Computer Science, University
of Michigan, Ann Arbor, MI 48109, USA cuixf@umich.edu, avestruz@umich.edu.}
}
\IEEEoverridecommandlockouts
\IEEEpubid{\makebox[\columnwidth]{\hfill} \hspace{\columnsep}\makebox[\columnwidth]{ }}
\maketitle
\IEEEpubidadjcol

\begin{abstract}
\hspace{0pt}This paper develops the singular perturbation theory for a particular
discrete-time nonlinear system which models the saturating inductor buck
converter using cycle-by-cycle digital control.
\end{abstract}

\section{Introduction}
Singular perturbation theory \cite{tikhonov1952} is a well-known method for studying nonlinear systems with two well-separated time scales. The $\chi$ nonlinear system is a finite-dimensional, discrete-time nonlinear system which can be used to model a broad class of power electronics systems.

Although there exist several discrete versions of singular perturbation theory in the literature \cite{tzuu1999tac,bouyekhf1996reduced, litkouhi1984infinite, kafri1996stability, bouyekhf1997analysis}, to the best of authors' knowledge, there is no such theorem which can be directly applied to the $\chi$ nonlinear system. Therefore, in this paper, we establish singular perturbation theory for the $\chi$ nonlinear system. 
\\

\section{System description}
The $\chi$ \emph{nonlinear system} is a finite-dimensional, discrete-time nonlinear system whose state-space representation can be expressed as
\begin{align}
    \label{orgsysmodel1} x[n+1] &= x[n] + f(\mu z[n]), \\
    \label{orgsysmodel2} z[n+1] &= g(x[n], z[n], \mu z[n]),
\end{align}
where $\mu$ is called {\emph{perturbation parameter}}.
\\
\noindent Direct quote from \cite{xiao2019buck}:
\par
\begingroup
\leftskip4em
\rightskip\leftskip 
We assume the following assumptions hold for all $(n,x,z) \in [0,\infty) \times D_x \times D_z$ containing the origin for some domain $D_x \subset \mathbb{R}^n$ and $D_z \subset \mathbb{R}^m$:
\begin{enumerate*}[label=(\alph*)]
    \item the matrix $\partial g / \partial z - I_{m}$  is invertible;
    \item the function $f$ and $g$ are Lipschitz in $(x,z)$ with Lipschitz constant $L_f$ and $L_g$;
    \item $f(0) = 0$, $g(0,0,0) = 0$.
\end{enumerate*}

From assumption (a) above and the implicit function theorem \cite{edwards2012advanced}, the equation $z = g(x, z, 0)$ has explicit solution $z = h(x)$. We assume the function $h$ is Lipschitz in $x$ with Lipschitz constant $L_h$. We define the \emph{reduced model} by
\begin{align} \label{redumodel1}
     x_s[n+1] &= x_s[n] + f(\mu h(x_s[n])), \\
     z_s[n] &= h(x_s[n]).
\end{align}
The reduced model describes trajectories of $x$ and $z$ which an observer sees in the slow time frame when $\mu$ approaches 0.
We define the \emph{boundary-layer model} by
\begin{align} \label{boundumodel1}
     y[n+1] =& g(x_s[n], y[n] \nonumber \\
     &+ h(x_s[n]),0) - h(x_s[n]).
\end{align}
where $y[n] = z_f[n] - z_s[n]$. The $z_f[n]$ is the trajectory of $z$ which an observer sees in the fast time frame when $\mu$ approaches 0. The boundary-layer model describes the difference between the trajectory of $z$ which an observer sees in the fast time frame and that in the slow time frame. We note that $y(n) = 0$ is a solution for the boundary-layer model.
\par
\endgroup
\section{Theory}
The following Theorem \ref{Theorem:TractoryConvergence} shows the relationship between the trajectory of the original system and that of the reduced model as well as the boundary-layer model.\\
Direct quote from \cite{xiao2019buck}:
\begin{theorem}
\label{Theorem:TractoryConvergence}
\par
\begingroup
\leftskip4em
\rightskip\leftskip
If $x_s = 0$ is an exponentially stable equilibrium of system (\ref{redumodel1}) and $y = 0$ is an exponentially stable equilibrium of system (\ref{boundumodel1}), uniform in $x_s$, then there exists a positive constant $\mu^*$ such that for all $n \ge 0$ and $0< \mu \le \mu^{*}$, the singular perturbation problem (\ref{orgsysmodel1}) and (\ref{orgsysmodel2}) has a unique solution $x[n,\mu]$, $z[n, \mu]$ on $[0,\infty)$, and
$ x[n, \mu] - x_s[n, \mu] = O(\mu), z[n,\mu] - h(x_s[n,\mu]) - y[n] = O(\mu)$ hold uniformly for $n \in [0, \infty)$, where $x_s[n, \mu]$ and $y[n]$ are the solutions of the system (\ref{redumodel1}) and system (\ref{boundumodel1}). Furthermore, there exists $n_1 > 0$, such that $z[n,\mu] - h(x_s[n]) = O(\mu)$ holds uniformly for $n \in [n_1, \infty)$.
\par
\endgroup
\end{theorem}

\begin{proof}
We use mathematic induction to prove \footnote{In this paper, we assume the 2-norm $\|\cdot\| = \|\cdot\|_2$.}
\begin{align} \label{eqn:approxx}
\norm{x[n]-x_s[n]} \le \lambda \mu.
\end{align}
Equation (\ref{eqn:approxx}) holds when $n=1$ because
\begin{align}
    \norm{x[1]-x_s[1]} = \norm{f(\mu z_0)} \le L \norm{z_0} \mu.
\end{align}
We suppose that there exists $\lambda$ satisfying $\norm{x[n]-x_s[n]} \le \lambda \mu $ when $n = k$. The following derivations prove (\ref{eqn:approxx}) when $n=k+1$:
\begin{align}
    &\norm{x[k+1]-x_s[k+1]} \nonumber \\
     = & \, \norm{x[k] + f(\mu z[k])-x_s[k] - f(\mu x_s[k])} \nonumber \\
     \le & \, \norm{x[k]-x_s[k]} + L_f \mu \norm{z[k] - x_s[k]} \nonumber \\
     \le & \, \lambda \mu + L_f \mu \norm{y(k)} \nonumber \\
      = & \, (\lambda + L_f \beta ) \mu.
\end{align}
By defining $ \lambda^*  = (\lambda + L_f \beta)$, we show $\norm{x[k+1]-x_s[k+1]} = O(\mu) $.\\
By mathematic induction, we conclude that
\begin{align}
    \norm{x[n] - x_s[n]} = O(\mu).
\end{align}
Then we prove  $z[n,\mu] - h(x_s[n]) - y[n] = O(\mu)$.
\begin{align}
    &\norm {z[n] - h(x_s[n]) - y[n]} \nonumber \\
    = \; & \|g(x[n-1], z[n-1], \mu(z[n-1]) - h(x_s[n]) \nonumber \\
    & - g(x_s[n-1], y[n-1] \nonumber \\
    & + h(x_s[n-1]),0) + h(x_s[n-1])\| \nonumber \\
        \le \; & \| g(x[n-1], z[n-1], \mu(z[n-1])  \nonumber \\
    &- g(x[n-1], z[n-1], 0) \| \nonumber \\
    \; & + \; \norm{g(x[n-1], z[n-1], 0) - g(x_s[n-1], z[n-1], 0)} \nonumber \\
    \; & + \; \norm{ h(x_s[n-1]) - h(x_s[n])} \nonumber \\
    \le \; &L_g \mu \norm{z[n-1]} + L_g \norm{x[n-1]-x_s[n-1]} \nonumber \\
    & +  L_h L_f \mu \norm{x_s[n-1]} \nonumber \\
    \le \; &L_g \mu \; \left(\norm{x_s[n-1]}+ \norm{y[n-1]}\right) \nonumber \\
    &  +L_g \norm{x[n-1]-x_s[n-1]}  +  L_h L_f \mu \norm{x_s[n-1]} \nonumber \\
    \le \; & \left(L_g( \lambda + \alpha + \beta ) + L_hL_f\alpha \right)\mu.
\end{align}
We conclude that
\begin{align}
    z[n,\mu] - h(x_s[n]) - y[n] = O(\mu).
\end{align}
Finally, we prove that there exists $n_1 > 0$, such that
\begin{align}
    z[n] - h(x[n]) = O(\mu)
\end{align}
holds uniformly for $n \in [n_1, \infty)$.\\
Let $n_1 = -\text{ln}(\mu)/\theta$, for all $n > n_1 $, 
\begin{align}
    \norm{y[n]} \le \norm{y[n_1]} = \epsilon \norm{y[0]} e^{\theta n_1} = \epsilon \norm{y[0]} \mu.
\end{align}
We proved $\norm{y[n]} = O(\mu)$ for all $n > n_1 $. Therefore
$ z[n] - h(x[n]) = O(\mu)$ holds uniformly for $n \in [n_1, \infty)$.
\end{proof}

The following Theorem \ref{Theorem:Stability} shows the relationship between the stability of the original system and that of the reduced model as well as the boundary-layer model.
\begin{theorem}
\label{Theorem:Stability}
 There exists $\mu^* > 0$ such that for all $\mu \le \mu^*$, then $x = 0$, $z = 0$ is an exponentially stable equilibrium of the singular perturbation problem (\ref{orgsysmodel1}) and (\ref{orgsysmodel2}).
\end{theorem}
\begin{proof}
The exponential stability of the $x = 0$ of system (\ref{redumodel1}) is equivalent to
\begin{align}
    \label{eqn:lyapfuncx1} &\gamma_1 \norm{x}^2 \le V(x) \le \gamma_2 \norm{x}^2, \\
    \label{eqn:lyapfuncx2} &V(x[n+1]) \le \sigma_1 V(x[n]),
\end{align}
where $\gamma_1>0$, $\gamma_2>0$, $0 < \sigma_1 < 1$.
The exponential stability of the equilibrium $y = 0$ of system (\ref{boundumodel1}) requires the following equations to hold uniformly in $x$
\begin{align}
    \label{eqn:lyapfuncy1} &\gamma_3 \norm{y}^2 \le W(y) \le \gamma_4 \norm{y}^2,\\
    \label{eqn:lyapfuncy2} &W(y[n+1]) \le \sigma_2 W(y[n]),
\end{align}
where $\gamma_3>0$, $\gamma_4>0$, $0 < \sigma_2 < 1$.\\
We use $\nu = V + W$ as a Lyapunov function candidate for the system (\ref{redumodel1}) and (\ref{boundumodel1}). From (\ref{eqn:lyapfuncx1}), (\ref{eqn:lyapfuncx2}), (\ref{eqn:lyapfuncy1}) and (\ref{eqn:lyapfuncy2})
\begin{align}
   \label{eqn:lyapfuncz1} &\text{min}\{\gamma_1, \gamma_3\} \norm{[x,y]}^2 \le \nu(x,y)\le \text{max}\{\gamma_2, \gamma_4\} \norm{[x,y]}^2, \\
   \label{eqn:lyapfuncz2} &\nu(x[n+1], y[n+1]) \le \text{max}\{\sigma_1, \sigma_2\}\nu(x[n], y[n]).
\end{align}
From (\ref{eqn:lyapfuncz1}) and (\ref{eqn:lyapfuncz2}), the equilibrium $[x,y] = 0$ of the systems (\ref{redumodel1}) and (\ref{boundumodel1}) is exponentially stable. Considering Theorem \ref{Theorem:TractoryConvergence}, we can conclude that the equilibrium $[x,y] = 0$ of the systems (\ref{orgsysmodel1}) and (\ref{orgsysmodel2}) is exponentially stable.
A more rigorous proof can be performed by following the same methods as the proofs of Proposition 8.1 and Proposition 8.2 in \cite{bof2018lyapunov}. 
\end{proof}

\section{Conclusion}
The theoretical contribution of this paper is developing singular perturbation theory for the $\chi$ nonlinear system.

\bibliographystyle{ieeetr}

\end{document}